\documentclass[12pt]{article}
\usepackage{amssymb}
\usepackage{amsmath}
\usepackage{amsfonts}
\usepackage{amsthm}

\def\pd#1#2{\frac{\partial#1}{\partial#2}}
\def\di{\bigstar}

\newcommand{\bea}{\begin{eqnarray}}
\newcommand{\eea}{\end{eqnarray}}

\theoremstyle{plain}

\newtheorem{theorem}{Theorem}

\newtheorem{proposition}[theorem]{Proposition}
\newtheorem{lemma}[theorem]{Lemma}

\theoremstyle{definition}

\newtheorem{definition}[theorem]{Definition}
\newtheorem{note}[theorem]{Note}

\begin{document}

\centerline{\Large {\bf Lie families: theory and applications}} \vskip 0.75cm

\centerline{ Jos\'e F. Cari\~nena$^{\dagger}$, Janusz Grabowski$^{\ddagger}$
 and Javier de Lucas$^{\ddagger}$}
\vskip 0.5cm

\centerline{$^{\dagger}$Departamento de  F\'{\i}sica Te\'orica, Universidad de Zaragoza,}
\medskip
\centerline{50009 Zaragoza, Spain.}
\medskip
\centerline{$^{\ddagger}$Institute of Mathematics, Polish Academy of Sciences,}
\medskip
\centerline{ul. \'Sniadeckich 8, P.O. Box 21, 00-956 Warszawa, Poland}
\medskip
\centerline{}
\medskip

\vskip 1cm

\begin{abstract}
We analyze families of non-autonomous systems of first-order ordinary differential equations admitting a {\it
common time-dependent superposition rule}, i.e., a time-dependent map expressing any solution of each of these
systems in terms of a generic set of particular solutions of the system and some constants. We next study
relations of these families, called {\it Lie families}, with the theory of Lie and quasi-Lie systems and apply
our theory to provide common time-dependent superposition rules for certain Lie families.




\end{abstract}

\section{Introduction.}
The theory of Lie systems \cite{LS, Ve93, Ve99, Gu93, PW, Ib99, CGM07} deals with non-autonomous systems of
first-order ordinary differential equations such that all their solutions can be written in terms of generic
sets of particular solutions and some constants, by means of a time-independent function. Such
functions are called {\it superposition rules} and the systems admitting this mathematical property are called
{\it Lie systems}. Lie succeeded in characterizing systems admitting a superposition rule. His result, known
now as {\it Lie Theorem} \cite{LS}, states that a non-autonomous system (time-dependent vector field) $X_t$ is
a Lie system if and only if there exists a finite-dimensional Lie algebra of vector fields $V_0$ such that
$X_t\in V_0$ for all $t$.

Note that a superposition rule can be found explicitly even for systems whose general solution is not known,
like in the case of Riccati equations \cite{CLR07b}, and its knowledge enables us to obtain the general
solution out of certain sets of particular solutions in an easier way than solving directly the system.

In the theory of Lie systems various methods have been developed to obtain superposition rules, time-dependent
and time-independent constants of the motion, exact solutions, integrability conditions, and other interesting
properties for particular systems \cite{HWA83, CLR08c, CL08M, OlmRodWin86, OlmRodWin87,
BecHusWin86,BecHusWin86b}. Unfortunately, being a Lie system is rather exceptional and, in order to apply the
methods of the theory of Lie systems to a broader set of non-autonomous systems, some generalizations of this
theory have been proposed. The generalized methods are presently used to investigate some partial differential
equations \cite{CGM07}, a class of second-order differential equations (the so-called {\it SODE Lie systems}
\cite{CLR08c}), certain Schr\"{o}dinger equations \cite{CarRam05b}, etc.

With the same aim of applying the theory of Lie systems to a broader family of systems, it has been recently
developed the theory of {\it quasi-Lie schemes} and  {\it quasi-Lie systems}
\cite{CGL08,CL08Diss,CLL08Emd}. This theory allows us to investigate some non-Lie systems and it can be
applied to dealing with certain second- and even higher-order systems of differential equations. For example,
it enables us to analyze some non-linear oscillators \cite{CGL08}, dissipative Milne--Pinney equations
\cite{CL08Diss}, Emden-Fowler equations \cite{CLL08Emd}, etc. One of the main results obtained through {\it
quasi-Lie scheme} approach is the existence of the so-called {\it time-dependent superposition rules}, that
is, time-dependent superposition functions expressing the general solution in terms of a generic family of
particular solutions of this system.

Note however that the concept of time-dependent superposition rule does not make much sense for a single
non-autonomous system. This is because, as explained in \cite{CGL08}, any single non-autonomous system admits
such a superposition rule which, however, can be as difficult to finding as the general solution of the system
and, therefore, it cannot be generally used to analyze properties of the system. This is analogous to the fact
that each autonomous system is automatically a Lie system, as each single vector field spans a one-dimensional
Lie algebra. Therefore, it only makes non-trivial sense to speak about common time-dependent superposition
rules for a bigger family of non-autonomous systems.

In this paper, we give a natural generalization of Lie Theorem characterizing Lie systems. This enables us to
show that many families of non-autonomous systems of first-order ordinary differential equations are {\it Lie
families}, that is they admit common time-dependent superposition rules. Furthermore, we study some Lie
families and we obtain common time-dependent superposition rules for all of them.

The organization of the paper goes as follows. Section 2 describes common time-dependent superposition rules
in terms of    certain horizontal foliations. In Section 3 we generalize Lie Theorem to characterize families
of systems admitting common time-dependent superposition rules. We posteriorly use this result to analyze the
relations between quasi-Lie systems, Lie systems and time-dependent superposition rules in Section 4. We
finally apply all our results to investigate some Lie families throughout Section 5.

\section{Time-dependent superpositions and foliations.}\label{TDS}

In this Section we develop the concept of common time-dependent superposition rule for a family of
non-autonomous systems of first-order ordinary differential equations and relate this concept to certain
horizontal foliations. For the sake of simplicity, we investigate these concepts in local coordinates, but our
approach can be slightly modified to handle systems on manifolds.

Consider a family, parametrized by elements $\alpha$ of a set $\Lambda$, of non-autonomous systems of
first-order ordinary differential equations on $\mathbb{R}^ n$ of the form
\begin{equation}\label{eq1}
\frac{dx^i}{dt}=Y^i_\alpha(t,x),\qquad i=1,\ldots,n,\qquad \alpha\in\Lambda.
\end{equation}
In applications, $\Lambda$ is often a finite subset of $\mathbb{ N}$ or $\Lambda=C^{\infty}(\mathbb{R})$.
Solutions of these systems are integral curves of the family $\{Y_\alpha\}_{\alpha\in\Lambda}$ of
time-dependent vector fields on $\mathbb{R}^n$ given by
\begin{equation}\label{FamVec}
Y_\alpha(t,x)=\sum_{i=1}^nY^i_\alpha(t,x)\pd{}{x^i},\qquad \alpha\in\Lambda.
\end{equation}
\begin{note}In order to simplify the terminology, we will use $Y_\alpha$ to designate both: a time-dependent
vector field of the above family and the non-autonomous system describing its integral curves.
\end{note}

Denote with $\bar Y_\alpha$ the autonomization of the time-dependent vector field $Y_\alpha$, that is, the
vector field on $\mathbb{R}\times\mathbb{R}^n$ defined by
$$
\bar Y_\alpha(t,x)=\frac{\partial}{\partial t}+\sum_{i=1}^nY^i_\alpha(t,x)\pd{}{x^i}.
$$
Integral curves of (\ref{eq1}) can be identified with trajectories of the vector field (autonomous system) $\bar Y_\alpha$.
Let us state the fundamental concept studied throughout the paper.

\begin{definition}\label{CSRDef} {\rm We say that the family of non-autonomous systems (\ref{eq1}) admits
a {\it common time-dependent superposition rule}, if there exists a map $\Phi:\mathbb{R}\times
\mathbb{R}^{n(m+1)}\rightarrow\mathbb{R}^{n}$,
\begin{equation}\label{TSupRul}
x=\Phi(t,x_{(1)},\ldots,x_{(m)};k_1,\ldots,k_n),
\end{equation}
such that the general solution $x(t)$ of any system $Y_\alpha$ of the family (\ref{eq1}) can be written, at
least for sufficiently small $t$, as
$$
x(t)=\Phi(t,x_{(1)}(t),\ldots,x_{(m)}(t);k_1,\ldots,k_n),
$$
with $\{x_{(a)}(t)\,|\,a=1,\ldots,m\}$ being a generic set of particular solutions of $Y_\alpha$, and
$k_1,\ldots,k_n$ being constants associated with each particular solution. A family of systems (\ref{eq1})
admitting a common time-dependent superposition rule is called a {\it Lie family}.}
\end{definition}

\begin{note}
We do not want to formalize precisely what 'generic' means in the above definition, as it is not crucial for
our purposes and depends on the context. One can have in mind the following example: for a system of linear
homogeneous differential equations 'generic' means that the particular solutions are linearly independent.
\end{note}

Given a common time-dependent superposition rule $\Phi:\mathbb{R}\times
\mathbb{R}^{n(m+1)}\rightarrow\mathbb{R}^n$ of a Lie family $\{Y_\alpha\}_{\alpha\in\Lambda}$, the map
$\Phi(t,x_{(1)},\ldots,x_{(m)};\cdot):\mathbb{R}^n\longrightarrow\mathbb{R}^n$,
$x_{(0)}=\Phi(t,x_{(1)},\ldots,x_{(m)};k)$, is regular for a generic point $(t,x_{(1)},\ldots, x_{(m)}) \in
\mathbb{R}\times \mathbb{R}^{nm}$ and, in view of the Implicit Function Theorem, it can be inverted  to write
$$
k=\Psi(t,x_{(0)},\ldots, x_{(m)}),
$$
for a certain map $\Psi: \mathbb{R}\times\mathbb{R}^{n(m+1)}\rightarrow \mathbb{R}^n$, and
$k=(k_1,\ldots,k_n)$ being the only point in $\mathbb{R}^n$ such that
$$
x_{(0)}=\Phi(t,x_{(1)},\ldots,x_{(m)};k).
$$
\begin{note}
As a matter of fact, the maps $\Psi$ and $\Phi$ are defined only locally on open subsets of
$\mathbb{R}\times\mathbb{R}^{n(m+1)}$ but, for simplicity, we will write $\mathbb{R}\times\mathbb{R}^{n(m+1)}$
for their domains.
\end{note}
Consequently, the map $\Psi$ determines locally a $n$-codimensional foliation $\mathfrak{F}$ of the manifold
$\mathbb{R}\times\mathbb{R}^{n(m+1)}$ into the level sets of $\Psi$. Moreover, as the fundamental property of
the map $\Psi$ establishes that $\Psi(t,x_{(0)}(t),\ldots, x_{(m)}(t))$ is constant for any $(m+1)$-tuple of
particular solutions of any system of the family (\ref{eq1}), the foliation determined by $\Psi$ is invariant
under the permutation of its $(m+1)$ arguments $\{x_{(a)}\,|\,a=0,\ldots,m\}$, and, differentiating
$\Psi(t,x_{(0)}(t),\ldots, x_{(m)}(t))$ with respect to $t$, we get
\begin{equation}\label{Psifi}
\pd{\Psi^j}{t}+\sum_{a=0}^{m}\sum_{i=1}^nY_{\alpha}^i(t,x_{(a)}(t))\pd{\Psi^j}{x^i_{(a)}}=0,\qquad\qquad
j=1,\ldots,n,\quad \alpha\in\Lambda,
\end{equation}
where $\Psi=(\Psi^1,\ldots,\Psi^n)$.

\begin{definition}
{\rm Given a time-dependent vector field  $Y=\sum_{i=1}^nY^i(t,x)\partial/\partial x^i$ on $\mathbb{R}^n$, we
define its {\it prolongation} to $\mathbb{R}\times\mathbb{R}^{n(m+1)}$ as the vector field on
$\mathbb{R}\times\mathbb{R}^{n(m+1)}$ given by
\begin{equation}\label{prolongation}
\widehat Y(t,x_{(0)},\ldots,x_{(m)})=\sum_{a=0}^{m}\sum_{i=1}^nY^ i(t,x_{(a)})\pd{}{x^ i_{(a)}},
\end{equation}
and its {\it time-prolongation} to $\mathbb{R}\times\mathbb{R}^{n(m+1)}$ as the vector field on
$\mathbb{R}\times\mathbb{R}^{n(m+1)}$ of the form
$$
\widetilde Y(t,x_{(0)},\ldots,x_{(m)})=\pd{}{t}+\sum_{a=0}^{m}\sum_{i=1}^nY^ i(t,x_{(a)})\pd{}{x^ i_{(a)}}.
$$}
\end{definition}
The equalities (\ref{Psifi}) show that the functions $\{\Psi^i\,|\,i=1,\ldots,n\}$ are first-integrals for the
vector fields $\{\widetilde Y_\alpha\}_{\alpha\in\Lambda}$, that is, $\widetilde Y_\alpha\Psi^ i=0$ for
$i=1,\ldots,n$ and $\alpha\in\Lambda$. Therefore, the vector fields $\widetilde Y_\alpha$ are tangent to the
leaves of $\mathfrak{F}$.

The foliation $\mathfrak{F}$ has another important property. If the leaf $\mathfrak{F}_k$ is the level set of
$\Psi$ corresponding to a certain $k=(k_1,\ldots,k_n)\in\mathbb{R}^n$, and given
$(t,x_{(1)},\ldots,x_{(m)})\in\mathbb{R}\times\mathbb{R}^{mn}$, there is only one point
$x_{(0)}\in\mathbb{R}^n$ such that  $(t,x_{(0)},x_{(1)},\ldots,x_{(m)})\in\mathfrak{F}_k$. Thus, the
projection onto the last $m\cdot n$ coordinates and the time
$$
\pi:(t,x_{(0)},\ldots,x_{(m)})\in\mathbb{R}\times\mathbb{R}^{n(m+1)}\longrightarrow
(t,x_{(1)},\ldots,x_{(m)})\in \mathbb{R}\times\mathbb{R}^{nm},$$ induces a local diffeomorphism from the leaf
$\mathfrak{F}_k$ of $\mathfrak{F}$ into $\mathbb{R}\times\mathbb{R}^{nm}$. We will say that the foliation
$\mathfrak{F}$ is {\it horizontal}  with respect to the projection $\pi$.

On the other hand, the horizontal foliation defines the common time-dependent superposition rule without
referring to the map $\Psi$. Indeed, if we take a point $x_{(0)}$ and $m$ particular solutions,
$x_{(1)}(t),\ldots,x_{(m)}(t)$, for a system of the family,  then $x_{(0)}(t)$ is the unique curve in
$\mathbb{R}^n$ such that the points of the curve
$$(t,x_{(0)}(t),x_{(1)}(t),\ldots, x_{(m)}(t))\subset
\mathbb{R}\times\mathbb{R}^{nm}$$ belong to the same leaf as the point
$(0,x_{(0)}(0),x_{(1)}(0),\ldots,x_{(m)}(0))$. Thus, it is only the horizontal foliation $\mathfrak{F}$ that
really matters when the common time-dependent superposition rule is concerned. It is in a sense obvious, as
composing $\Psi$ with a diffeomorphism on $\mathbb{R}^n$ changes the superposition function (rearranges the level sets) but yields the same superposition rule. This proves the
following (cf. \cite{CGM07}).

\begin{proposition}\label{Connection}
Giving a common time-dependent superposition rule (\ref{TSupRul}) for a Lie family (\ref{eq1}) is equivalent
to giving a foliation which is horizontal with respect to the projection
$\pi:\mathbb{R}\times\mathbb{R}^{(m+1)n}\rightarrow\mathbb{R}\times\mathbb{R}^{nm}$ and such that the vector
fields $\{\widetilde Y_\alpha\}_{\alpha\in\Lambda}$ are tangent to their leaves.
\end{proposition}

\section{Generalized Lie Theorem.}\label{GLT}

It is generally difficult to determine whether a family (\ref{family}) admits a common time-dependent
superposition rule by means of Proposition \ref{Connection}. It is therefore interesting to find a
characterization of Lie families by means of a more convenient criterion, e.g. through an easily verifiable
condition based on the properties of the time-dependent vector fields $\{Y_\alpha\}_{\alpha\in\Lambda}$.
Finding such a criterion is the main result of this section. It is formulated as Generalized Lie Theorem.

\noindent We start with three lemmata. The proofs of first two of them are straightforward.
\begin{lemma}\label{PureProl} Given two time-dependent vector fields $X$ and $Y$ on $\mathbb{R}^n$,
the commutator $[\widetilde X,\widetilde Y]$ on $\mathbb{R}\times\mathbb{R}^{n(m+1)}$ is the prolongation
of a time-dependent vector field $Z$ on $\mathbb{R}^n$, $[\widetilde X,\widetilde Y]=\widehat Z$.

\end{lemma}
\begin{lemma}\label{Aut} Given a family of time-dependent vector fields, $X_1,\ldots,X_r$,  on $\mathbb{R}^n$,
their autonomizations satisfy the relations
$$
[\bar X_j, \bar X_k](t,x)=\sum_{l=1}^rf_{jkl}(t)\bar X_l(t,x),\qquad j,k=1,\ldots,r,
$$
for some time-dependent functions $f_{jkl}:\mathbb{R}\rightarrow\mathbb{R}$, if and only if their
time-prolongations to $\mathbb{R}\times\mathbb{R}^{n(m+1)}$, $\widetilde X_1,\ldots,\widetilde X_r$, satisfy
analogous relations
$$
[\widetilde X_j, \widetilde X_k](t,x)=\sum_{l=1}^rf_{jkl}(t)\widetilde X_l(t,x),\qquad j, k=1,\ldots,r.
$$
Moreover, $\sum_{l=1}^rf_{jkl}(t)=0$ for all $j, k=1,\ldots,r$.
\end{lemma}

\begin{lemma}\label{Prol} Consider a family of time-dependent vector fields, $Y_1,\ldots,Y_r$, with time prolongations
to $\mathbb{R}\times\mathbb{R}^{n(m+1)}$, $\widetilde Y_1,\ldots,\widetilde Y_r$, such that their projections
$\pi_*(\widetilde Y_j)$ are linearly independent at a generic point in $\mathbb{R}\times\mathbb{R}^{nm}$.
Then, $\sum_{j=1}^r b_j\widetilde Y_j$, with $b_j\in C^{\infty}(\mathbb{R}\times\mathbb{R}^{nm})$, is of the form $\widehat Y$ (resp. $\widetilde Y$) for a time-dependent vector field $Y$ on
$\mathbb{R}^n$, if and only if the functions $b_j$ depend on the time only, that is, $b_j=b_j(t)$, and
$\sum_{j=1}^rb_j=0$ (resp., $\sum_{j=1}^rb_j=1$).
\end{lemma}
\begin{proof}
We shall only detail the proof of the above claim for $\sum_{j=1}^r b_j\widetilde Y_j=\widehat Y$, as the proof of the other case is completely analogous. Let us write in coordinates
\begin{equation*}
    \widetilde Y_j=\frac{\partial}{\partial t}+\sum_{i=1}^n\sum_{a=0}^mA^i_j(t,x_{(a)})\pd{}{x^i_{(a)}},\qquad j=1,\ldots, r.
 \end{equation*}
Then, {\small
\begin{multline*}
\sum_{j=1}^rb_j(t,x_{(0)},\ldots,x_{(m)})\widetilde Y_j=\sum_{j=1}^r\sum_{i=1}^n\sum_{a=0}^mb_j(t,x_{(0)},
\ldots,x_{(m)})A^i_j(t,x_{(a)})\pd{}{x^i_{(a)}}\\
+\sum_{j=1}^r b_j(t,x_{(0)},\ldots,x_{(m)})\pd{}{t},
\end{multline*}}which is a prolongation if and only if there are functions $B^i:\mathbb{R}\times\mathbb{R}^n\rightarrow\mathbb{R}$,
with $i=1,\ldots,n$, such that
{\small
\begin{equation*}
\left\{\begin{aligned}
&\sum_{j=1}^r b_j(t,x_{(0)},\ldots,x_{(m)})A^i_j(t,x_{(a)})=B^i(t,x_{(a)}),\\
&\sum_{j=1}^r b_j(t,x_{(0)},\ldots,x_{(m)})=0,
\end{aligned}\right.\qquad a=0,\ldots,m, \quad i=1,\ldots,n.
\end{equation*}}If the functions $b_1,\ldots,b_r$ are time-dependent only and
$\sum_{j=1}^r b_j=0$, the above conditions hold and $\sum_{j=1}^r b_j \widetilde Y_j$ is the prolongation to
$\mathbb{R}\times\mathbb{R}^{n(m+1)}$ of the time-dependent vector field
$Y=\sum_{i=1}^nB^i(t,x)\partial/\partial x^i.$

Conversely, suppose that $\sum_{j=1}^r b_j \widetilde Y_j$ is a prolongation for a time-dependent vector field
on $\mathbb{R}^n$.  In this case, the functions $b_j(t,x_{(0)},\ldots,x_{(m)})$ solve the following system of
linear equations in the unknown variables $u_\alpha$:
\begin{equation*}
\left\{\begin{aligned}
&\sum_{j=1}^r u_j A^i_j(t,x_{(a)})=B^i(t,x_{(a)}),\\
&\sum_{j=1}^r u_j=0,
\end{aligned}\right.
\end{equation*}
where $a=1,\ldots,m$ and $i=1,\ldots,n$. This is a system of $m\cdot n+1$ equations and, as $\pi_*(\widetilde
Y_j)$, with $j=1,\ldots,r$, are linearly independent by assumption, the solutions $u_\alpha$ are uniquely
determined by the variables $\{t, x_{(1)},\ldots,x_{(m)}\}$ and therefore they do not depend on $x_{(0)}$.
Since time-prolongations are invariant with respect to the symmetry group $S_{m+1}$ acting on
$\mathbb{R}^{n(m+1)}=\left(\mathbb{R}^n\right)^{m+1}$ in the obvious way, the functions
$b_j(t,x_{(0)},\ldots,x_{(m)})$, with $j=1,\ldots,r$, must satisfy such a symmetry. Hence, as they do not
depend on $x_{(0)}$, they cannot depend on the variables $\{x_{(1)},\ldots,x_{(m)}\}$ and they are functions
depending on the time only.
\end{proof}

\begin{theorem}{\bf (Generalized Lie Theorem)} \label{MT}
The family of systems (\ref{eq1}) admits a {\it common time-dependent superposition rule} if and only if the
vector fields $\{\bar Y_\alpha\}_{\alpha\in\Lambda}$ can be written in the form
\begin{equation}\label{MainDesc}
\bar Y_\alpha(t,x)=\sum_{j=1}^rb_{\alpha j}(t)\bar X_j(t,x),\qquad \alpha\in\Lambda,
\end{equation}
where $b_{\alpha j}$ are functions of the time only, $\sum_{j=1}^nb_{\alpha j}=1$, and, $X_1,\ldots, X_r$, are
time-dependent vector fields such that
\begin{equation}\label{condition}
[\bar X_j,\bar X_k](t,x)=\sum_{l=1}^rf_{jkl}(t)\bar X_l(t,x),\qquad j, k=1,\ldots,r.
\end{equation}
for some functions $f_{jkl}:\mathbb{R}\rightarrow\mathbb{R}$, with $j,k,l=1,\ldots,r$. We call the family of
autonomizations, $\bar X_1,\ldots,\bar X_r$, a {\sl system of generators of the Lie family}.
\end{theorem}
\begin{proof}

Suppose first that the family of systems (\ref{eq1}) admits a {\it common time-dependent superposition rule}
and let $\mathfrak{F}$ be the corresponding $n$-codimensional horizontal foliation. The vector fields
$\{\widetilde Y_\alpha\}_{\alpha\in\Lambda}$ are tangent to the leaves of the foliation $\mathfrak{F}$ and
span a distribution $\mathcal{D}_0$ on $\mathbb{R}\times\mathbb{R}^{n(m+1)}$. Such a distribution need not be
involutive, see examples in Section \ref{Examples}. Nevertheless, we can enlarge the family $\{\widetilde
Y_\alpha\}_{\alpha\in\Lambda}$ to the Lie algebra of vector fields generated by such a family. This Lie
algebra is spanned by $\{\widetilde Y_\alpha\}_{\alpha\in\Lambda}$ and all their possible Lie brackets, i.e.,
\begin{equation}\label{family}\widetilde Y_\alpha, \,\, [\widetilde Y_\alpha,\widetilde Y_\beta],\,\,
[\widetilde Y_\alpha,[\widetilde Y_\beta,\widetilde Y_\gamma]],\,\,[\widetilde Y_\alpha,[\widetilde Y_\beta,
[\widetilde Y_\gamma,\widetilde Y_\delta]]],\ldots\,\,\qquad\qquad \alpha,\beta,\gamma,\delta,\ldots\in \Lambda.
\end{equation}
All the above vector fields are tangent to the leaves of the foliation $\mathfrak{F}$ and therefore there are
up to $m\cdot n+1$ linearly independent ones at a generic point of $\mathbb{R}\times\mathbb{R}^{n(m+1)}$.
Consequently, they span an involutive generalized distribution $\mathcal{D}$ with leaves of dimension $r\leq
m\cdot n+1$. In a neighborhood of a regular point of this foliation, take now a finite basis of vector fields
from the elements of the family (\ref{family}) spanning the distribution. By construction, at least one of
them must be of the form $\widetilde X_1$ for a certain time-dependent vector field $X_1$ on $\mathbb{R}^n$
and, in view of lemma \ref{Prol} and the form of the family (\ref{family}), those not being time-prolongations
are just prolongations. Therefore, if we add $\widetilde X_1$ to those elements of the basis being
prolongations, we get a new basis of the distribution $\mathcal{D}$ made up by certain $r\leq m\cdot n+1$
time-prolongations $\widetilde X_1,\ldots,\widetilde X_r$. In other words, the distribution $\mathcal{D}$ is
locally spanned, near regular points, by time-prolongations, say $\widetilde X_1,\ldots,\widetilde X_r$. As
the generalized distribution $\mathcal{D}$ is involutive, there exist $r^3$ real functions $f_{jkl}$, with
$j,k,l=1,\ldots,r,$ on $\mathbb{R}\times\mathbb{R}^{n(m+1)}$ such that
$$
[\widetilde X_j,\widetilde X_k]=\sum_{l=1}^rf_{jkl}\widetilde X_l, \qquad j,k=1,\ldots,r,
$$
and as the left side of the above equalities are prolongations, we get that, in view of lemma \ref{Prol}, all
the functions $f_{jkl}$ depend on time only and $\sum_{l=1}^nf_{jkl}=0$. Finally, taking into account Lemma
\ref{Aut}, we have
$$
[\bar X_j,\bar X_k](t,x)=\sum_{l=1}^rf_{jkl}(t)\bar X_l(t,x), \qquad j,k=1,\ldots,r.
$$
Note that as the vector fields $\{\widetilde Y_\alpha\}_{\alpha\in\Lambda}$ are contained in the distribution $\mathcal{D}$, there exist some functions $b_{\alpha j}\in C^{\infty}(\mathbb{R}^{n(m+1)})$ such that $\widetilde Y_{\alpha}=\sum_{j=1}^rb_{\alpha j}\widetilde X_j$ for every $\alpha\in\Lambda$. In consequence, according again to lemma \ref{Prol}, the functions $b_{\alpha j}$ depend on the time only, i.e., $b_{\alpha j}=b_{\alpha j}(t)$. Therefore, we get that
$$
\widetilde Y_\alpha=\sum_{j=1}^rb_{\alpha j}\widetilde X_j\Longrightarrow\bar Y_\alpha(t,x)=\sum_{j=1}^rb_{\alpha j}(t)\bar X_j(t,x),\qquad \alpha\in\Lambda.
$$

Let us prove the converse. Assume that we can write
$$
\bar Y_\alpha(t,x)=\sum_{j=1}^rb_{\alpha j}(t) \bar X_j(t,x)
$$
for certain time-dependent vector fields $X_1,\ldots,X_r$ on $\mathbb{R}^n$ such that
$$
[\bar X_j,\bar X_k](t,x)=\sum_{l=1}^rf_{jkl}(t)\bar X_l(t,x), \qquad j,k=1,\ldots,r.$$ In view of lemma
\ref{Aut}, the vector fields $\widetilde X_1,\ldots, \widetilde X_r$ span an involutive distribution
$\mathcal{D}$ on $\mathbb{R}\times\mathbb{R}^{n(m+1)}$ for any $m$. Furthermore, the rank of this distribution
is not greater than $r$ and therefore, for $m$ big enough, this distribution is at least $n$-codimensional and
it gives rise to a foliation $\mathfrak{F}_0$ which is horizontal with respect to the projection $\pi$.
Moreover, if codimension of $\mathfrak{F}_0$ is bigger than $n$, we can enlarge $\mathfrak{F}_0$ to a
$n$-codimensional foliation $\mathfrak{F}$, still horizontal with respect to the map $\pi$ giving rise to a
common time-dependent superposition rule for the family (\ref{eq1}).
\end{proof}

\section{Lie families, quasi-Lie and Lie systems}\label{QLSandLS}

This section is devoted to recalling the theories of quasi-Lie schemes and Lie systems needed to investigate
the relations among these theories and Lie families. A full detailed report on these topics can be found in
\cite{CGM07,CGL08}.

The theory of quasi-Lie schemes provides various results on the transformation properties of time-dependent
vector fields by a certain kind of time-dependent changes of variables associated with generalized flows.

Each time-dependent vector field $X$ gives rise to a {\it generalized flow} $g^X$, i.e., a map
$g^X:(t,x)\in\mathbb{R}\times \mathbb{R}^n \rightarrow  g^X_t(x)\equiv g^X(t,x)\in \mathbb{R}^n$ (more
precisely, defined in a neighborhood of $\{ 0\}\times\mathbb{R}^n$ in $\mathbb{R}\times \mathbb{R}^n$), with
$g^X_0={\rm Id}_{\mathbb{R}^n}$, such that the curve $\gamma^X_{x_0}(t)=g^X_t(x_0)$ is the integral curve of
the time-dependent vector field $X$ starting from the point $x_0\in\mathbb{R}^n$, i.e., $\dot
\gamma^X_{x_0}=X(t,\gamma^X_{x_0}(t))$ and $\gamma^X_{x_0}(0)=g^X_0(x_0)=x_0$.

Denote with $\mathfrak{X}_t(\mathbb{R}^n)$ the set of all time-dependent vector fields on $\mathbb{R}^n$. Each
generalized flow $h$ acts on the set of time-dependent vector fields $\mathfrak{X}_t(\mathbb{R}^n)$
transforming each time-dependent vector field $X\in\mathfrak{X}_t(\mathbb{R}^n)$ into a new one, $h_\di X$, with
the generalized flow of the form $g^{h_\di X}=h\circ g^X$. In terms of autonomizations we can write
(\cite[Theorem 3]{CGL08})
$$\overline{h_\di X}=\bar{h}_*\bar{X}\,,$$
where $\bar{h}$ is the natural {\it autonomization} of the generalized flow $h$ to a (local) diffeomorphism of
$\mathbb{R}\times \mathbb{R}^n$, $\bar{h}(t,x)=(t,h_t(x))$ and where $\bar{h}_*$ is the standard action of the
diffeomorphism $\bar{h}$ on vector fields. This, in turn, implies that
\begin{equation}\label{aut}
[\overline{h_\di X},\overline{h_\di Y}]=\bar{h}_*[\bar{X},\bar{Y}]\,.
\end{equation}

Let $V$ be a finite-dimensional vector space of vector fields on $\mathbb{R}^n$. We denote with
$V(\mathbb{R})$ the set of time-dependent vector fields $X\in\mathfrak{X}_t(\mathbb{R}^n)$ with values in $V$,
that is, those time-dependent vector fields $X$ such that, for every $t\in\mathbb{R}$, the vector field
$X_t(x)$ belongs to $V$. In terms of the introduced terminology and notation, Lie Theorem, whose  statement can be found for instance in \cite{LS,CGM07}, can be reformulated
as follows.

\begin{proposition} {\bf (Lie Theorem)} A non-autonomous system $X$ is a Lie system on $\mathbb{R}^n$ if and only
if there exists a finite-dimensional Lie algebra of vector fields $V_0\subset \mathfrak{X}(\mathbb{R}^n)$ such that $X\in V_0(\mathbb{R})$.
\end{proposition}

\begin{definition}
A {\em quasi-Lie scheme} $S(W,V)$ on the manifold $M$ consists of two finite-dimensional vector spaces of
vector fields $W,V\subset\mathfrak{X}(M)$ such that
\begin{itemize}
 \item $W$ is a linear subspace of $V$.
\item $W$ is a Lie algebra of vector fields, that is, $[W,W]\subset W$. \item $W$ normalizes $V$, i.e.,
$[W,V]\subset V$.
\end{itemize}
\end{definition}

 It has been proved in \cite{CGL08} that given a quasi-Lie scheme $S(W,V)$, the space $V(\mathbb{R})$
 is stable under the action of the infinite-dimensional group $\mathcal{G}(W)$ of generalized flows of
 vector fields in $W(\mathbb{R})$, i.e., $g_\di X\in V(\mathbb{R})$, for every $X\in V(\mathbb{R})$ and $g\in \mathcal{G}(W)$.

\begin{definition}
Given a quasi-Lie scheme $S(W,V)$, we say that a time-dependent vector field $X\in V(\mathbb{R})$ is a {\it
quasi-Lie system} with respect to this scheme, if there exist a generalized flow $g\in\mathcal{G}(W)$ and a
Lie algebra of vector fields $V_0\subset V$, such that $g_\di X\in V_0(\mathbb{R})$.
\end{definition}

As for each Lie system $X$ there exists a Lie algebra of vector fields $V_0$ such that $X\in V_0(\mathbb{R})$,
it is obvious that $X\in S(V_0,V_0)$ and, consequently, every Lie system is also a quasi-Lie system.

From now on, given a quasi-Lie scheme $S(W,V)$, a generalized flow $g\in\mathcal{G}(W)$, and a Lie algebra of
vector fields $V_0\subset V$, we denote with $S_g(W,V;V_0)$ the set of quasi-Lie systems of the scheme
$S(W,V)$ such that $g_\di X\in V_0(\mathbb{R})$.

\begin{proposition}\label{QL}
The family of quasi-Lie systems $S_g(W,V;V_0)$ is a Lie family admitting the common time-dependent
superposition function of the form
\begin{equation}\label{sf}
\bar{\Phi}_g(t,x_{(1)},\dots,x_{(m)},k)=g_t^{-1}\circ\Phi\left(g_t(x_{(1)}),\dots,g_t(x_{(m)}),k\right)\,,
\end{equation}
for any time-independent superposition function $\Phi$ associated with the Lie algebra of vector fields $V_0$
by Lie Theorem.
\end{proposition}

\begin{proof}
Let $Z_1,\ldots,Z_r$ be a basis in $V_0$. Since $V_0$ is closed with respect to the Lie bracket,
\begin{equation}\label{rel0}
[Z_j,Z_k]=\sum_{l=0}^rc_{jkl} Z_l
\end{equation}
for some constants $c_{jkl}$, $j,k,l=1,\ldots,r$. For any $Y\in S_g(W,V;V_0)$, there exist functions $b_j$
such that
$$
(g_{\di}Y)_t(x)=\sum_{j=1}^rb_j(t)Z_j(x)\,.
$$
Consequently, for the autonomization we can write
\begin{equation}\label{DeC}
\overline{g_{\di}Y}(t,x)=\sum_{j=0}^{r}b_j(t)\bar Z_j(t,x)\,,
\end{equation}
where we put $Z_0=0$ (thus $\bar Z_0=\partial/\partial t$) and $b_0(t)=1-\sum_{j=1}^{r}b_j(t)$.

Note that, as $Z_k$ are time-independent, the autonomizations $\bar Z_k$, $k=0,\dots,r$, form a Lie algebra:
\begin{equation}\label{rel}
[\bar Z_j,\bar Z_k]=\sum_{l=0}^rc_{jkl}\bar Z_l\,,
\end{equation}
where $c_{0kl}=c_{j0l}=0$ and $c_{jk0}=-\sum_{l=1}^rc_{jkl}$, for $j,k=1,\dots,r$. Hence, according to
(\ref{aut}), the autonomizations $\bar Z'_k=\overline{g_{\di}^{-1}(Z_k)}$ are also closed with respect to the bracket,
\begin{equation}\label{rel1}
[\bar Z'_j,\bar Z'_k]=\sum_{l=0}^rc_{jkl}\bar Z'_l\,,
\end{equation}
and, in view of (\ref{DeC}), the autonomization of any $Y\in S_g(W,V;V_0)$ can be written in the form
$$
\bar Y(t,x)=\sum_{j=0}^{r}b_j(t)\bar Z'_j(t,x)\,.
$$
This means, in view of Theorem \ref{MT}, that $S_g(W,V;V_0)$ is a Lie family. The form (\ref{sf}) can be now
easily derived (see \cite[Theorem 4]{CGL08}).
\end{proof}

In view of the above proposition, every quasi-Lie system and, consequently, every Lie system can be included
in a Lie family satisfying Theorem \ref{MT}. This fact  justifies once more calling this theorem Generalized
Lie Theorem.

\section{Applications.}\label{Examples}
 In this section we will apply common time-dependent superposition rules for studying some first- and second-order
 differential equations. In this way, we will show how that common time-dependent superposition rules can be used
 to analyze equations which cannot be studied by means of the usual theory of Lie systems. Additionally, some new results for the study of Abel and Milne--Pinney equations are provided.

\subsection{A time-dependent superposition rule for Abel equations}
We illustrate here our theory by deriving a common time-dependent superposition rule for a Lie family of Abel
equations whose elements do not admit a standard superposition rule except for a few particular instances. In
this way, we single out that our theory provides new tools for investigating solutions of non-autonomous
systems of differential equations than cannot be analyzed by means of the theory of Lie systems.

With this aim, we analyze the so-called Abel equations of the first-type \cite{Bo05,TR05}, i.e., the
differential equations of the form
\begin{equation}\label{Abel}
 \frac{dx}{dt}=a_0(t)+a_1(t)x+a_2(t)x^2+a_3(t)x^3,
\end{equation}
with $a_3(t)\neq 0$. Abel equations appear in the analysis of several cosmological models \cite{MML08,HM04,CK86} and other different fields in Physics \cite{ZT09,SP03,CLS04,GE04,Es02,CLP01}. Additionally, the study of integrability conditions for Abel equations is a research topic of current interest in Mathematics and multiple studies have been carried out in order to analyze the properties of the solutions of these equations \cite{TR05,Ch40,SR82,Al07,MCH01}.

Note that, apart from its inherent mathematical interest, the knowledge of particular solutions of Abel equations allows us to study the properties of those physical systems that such equations describe. Thus, the expressions enabling us to obtain easily new solutions of Abel equations by means of several particular ones, like common time-dependent superposition rules, are interesting to study the solutions of these equations and, therefore, their related physical systems.

Unfortunately, all the expressions describing the general solution of Abel equations presently known can only be applied to study autonomous instances and, moreover, they depend on families of particular conditions satisfying certain extra conditions, see \cite{Ch40,SR82}. Taking this into account, common time-dependent superposition rules represent an improvement with respect to these previous expressions, as they permit one to treat non-autonomous Abel equations and they do not require the use of particular solutions satisfying additional conditions.

Recall that, according to Theorem \ref{MT}, the existence of a common time-dependent superposition rule for a
family of time-dependent vector fields (\ref{FamVec}) requires the existence of a system of generators, i.e.,
a certain set of time-dependent vector fields, $X_1,\ldots, X_r$, satisfying relations (\ref{condition}).
Conversely, given such a set, the family of time-dependent vector fields $Y$ whose autonomizations can be
written in the form
$$\bar Y_\alpha(t,x)=\sum_{j=1}^rb_j(t)\bar X_j(t,x),\qquad \sum_{j=1}^rb_j(t)=1,$$
admits a common time-dependent superposition rule and becomes a Lie family.

Consequently, a Lie family of Abel equations can be determined, for instance, by finding two time-dependent
vector fields of the form
\begin{equation}\label{ansatz}
\begin{aligned}
X_1(t,x)&=(b_0(t)+b_1(t)x+b_2(t)x^2+b_3(t)x^3)\frac{\partial}{\partial x},\\
X_2(t,x)&=(b'_0(t)+b'_1(t)x+b'_2(t)x^2+b'_3(t)x^3)\frac{\partial}{\partial x},\qquad b'_3(t)\neq 0,
\end{aligned}
\end{equation}
such that
\begin{equation}\label{cond}
[\bar X_1,\bar X_2]=2(\bar X_2-\bar X_1).
\end{equation}

Let us analyze the existence of such two time-dependent vector fields $X_1$ and $X_2$ holding relation
(\ref{cond}). In coordinates, the Lie bracket $[\bar X_1,\bar X_2]$ reads
\begin{multline*}
[(b_3'b_2-b_2'b_3)x^4+(2(b_3'b_1-b_3b_1')-\dot b_3+\dot b_3')x^3+(-3(b_0'b_3-b_0b_3')+(b_2'b_1-b_2b_1')\\-\dot
b_2+\dot b_2')x^2+(-2b_0'b_2+2b_0b'_2-\dot b_1+\dot b_1')x-b_0'b_1+b_0b_1'-\dot b_0+\dot
b_0']\frac{\partial}{\partial x}.
\end{multline*}
Hence, in order to satisfy condition (\ref{cond}), $b_3'b_2-b_2'b_3=0$, e.g. we may fix $b_2=b_3=0$. Additionally, for the sake of simplicity, we assume $b_3'=1$. In this case, the previous expression takes the form
\begin{equation*}
[2b_1x^3+(3b_0+b_2'b_1+\dot b_2')x^2+(2b_0b'_2-\dot b_1+\dot b_1')x-b_0'b_1+b_0b_1'-\dot b_0+\dot
b_0']\frac{\partial}{\partial x},
\end{equation*}
and, taking into account the values chosen for $b_2$, $b_3$ and $b_3'$, assumption (\ref{cond}) yields $b_1=1$
and
\begin{equation*}\left\{
\begin{aligned}
 b_2'&=3b_0+\dot b_2',\\
 2(b_1'-1)&=2b_0b'_2+\dot b_1',\\
2(b_0'-b_0)&=-b_0'+b_0b_1'-\dot b_0+\dot b_0'.
\end{aligned}\right.
\end{equation*}
As the above system has more variables than equations, we can try to fix some values of the variables in order
to simplify it and obtain a particular solution. In this way, taking $b_0(t)=t$, the above system reads
\begin{equation*}\left\{
\begin{aligned}
 \dot b_2'&=b_2'-3t,\\
 \dot b_1'&=2(b_1'-1)-2tb'_2,\\
\dot b_0'&=2(b_0'-t)+b_0'-tb_1'+1.
\end{aligned}\right.
\end{equation*}
This system is integrable by quadratures and it can be verified that it admits the particular solution
$$b_2'(t)=3(1+t),\quad  b_1'(t)=3(1+t)^2+1,\quad  b_0'(t)=(1+t)^3+t.$$
Summing up, we have proved that the time-dependent vector fields
\begin{equation}\label{fami}
\left\{\begin{aligned} X_1(t,x)&=(t+x)\frac{\partial}{\partial x},\\
X_2(t,x)&=((1+t)^3+t+(3(1+t)^2+1)x+3(1+t)x^2+x^3)\frac{\partial}{\partial x},
\end{aligned}\right.
\end{equation}
satisfy (\ref{cond}) and, therefore, the family of time-dependent vector fields
$Y_{b(t)}(t,x)=(1-b(t))X_1(x)+b(t)X_2(x)$ is a Lie family. The corresponding family of Abel equations is
\begin{equation}\label{LieFamily}
\frac{dx}{dt}=(t+x)+b(t)(1+t+x)^3.
\end{equation}
According to the results proved in Section \ref{GLT}, in order to determine a common time-dependent
superposition rule for the above Lie family we have to determine a first-integral for the vector fields of the
distribution $\mathcal{D}$ spanned by the time-prolongations $\widetilde X_1$ and $\widetilde X_2$ on
$\mathbb{R}\times\mathbb{R}^{n(m+1)}$ for a certain $m$ so that the time-prolongations of $X_1$ and $X_2$ to
$\mathbb{R}\times \mathbb{R}^{nm}$ were linearly independent at a generic point. Taking into account
expressions (\ref{fami}), the prolongations of the vector fields $X_1$ and $X_2$ to
$\mathbb{R}\times\mathbb{R}^2$ are linearly independent at a generic point and, in view of (\ref{cond}), the
time-prolongations $\widetilde X_1$ and $\widetilde X_2$ to $\mathbb{R}\times\mathbb{R}^3$ span an involutive
generalized distribution $\mathcal{D}$ with leaves of dimension two in a dense subset of
$\mathbb{R}\times\mathbb{R}^3$. Finally, a first-integral for the vector fields in the distribution
$\mathcal{D}$ will provide us a common time-dependent superposition rule for the Lie family (\ref{LieFamily}).

Since, in view of (\ref{cond}), the vector fields $\widetilde X_1$ and $\widetilde X_2$ span the distribution
$\mathcal{D}$, a function $G:\mathbb{R}\times\mathbb{R}^2\rightarrow \mathbb{R}$ is a first-integral of the
vector fields of the distribution $\mathcal{D}$ if and only if $G$ is a first-integral of $\widetilde X_1$ and
$\widetilde X_1-\widetilde X_2$, i.e. $\widetilde X_1G=(\widetilde X_2-\widetilde X_1)G=0$.

The condition $\widetilde X_1G=0$ reads
$$
\frac{\partial G}{\partial t}+(t+x_0)\frac{\partial G}{\partial x_0}+(t+x_1)\frac{\partial G}{\partial x_1}=0,
$$
and, using the method of characteristics \cite{MetChar}, we note that the curves on which $G$ is constant, the
so-called {\it characteristics}, are solutions of the system
$$
dt=\frac{dx_0}{t+x_0}=\frac{dx_1}{t+x_1}\Rightarrow \frac{dx_i}{dt}=t+x_i,\qquad i=0,1,
$$
i.e., $x_i(t)=\xi_ie^{t}-t-1$, with $i=0,1$. These solutions are determined by the implicit equations
$\xi_0=e^{-t}(x_0+t+1)$ and $\xi_1=e^{-t}(x_1+t+1)$, with $\xi_0,\xi_1\in\mathbb{R}$. Therefore, there exists
a function $G_2:\mathbb{R}^2\rightarrow \mathbb{R}$ such that $G(t,x_0,x_1)=G_2(\xi_0,\xi_1)$. In other words,
each first-integral $G$ of $\widetilde X_1$ depends only on $\xi_0$ and $\xi_1$.

Taking into account the previous fact, we look for first-integrals of the vector field $\widetilde
X_2-\widetilde X_1$ being also first-integrals of $\widetilde X_1$, that is, for solutions of the equation
$(\widetilde X_2-\widetilde X_1)G=0$ with $G$ depending on $\xi_0$ and $\xi_1$. Using the expression of
$\widetilde X_2-\widetilde X_1$ in the system of coordinates $\{t, \xi_0,\xi_1\}$, we get that
$$
\xi_0^3\frac{\partial G}{\partial \xi_0}+\xi_1^3\frac{\partial G}{\partial \xi_1}=\xi_0^3\frac{\partial
G_2}{\partial \xi_0}+\xi_1^3\frac{\partial G_2}{\partial \xi_1}=0,
$$
and, applying again the method of characteristics, we obtain that there exists a function
$G_3:\mathbb{R}\rightarrow\mathbb{R}$ such that $G(t,x_0,x_1)=G_2(\xi_0,\xi_1)=G_3(\Delta)$, where
$\Delta=e^{2t}((x_0+t+1)^{-2}-(x_1+t+1)^{-2})$. Finally, using this first-integral, we get that the common
time-dependent superposition rule for the Lie family (\ref{LieFamily}) reads
$$
k=e^{2t}((x_0+t+1)^{-2}-(x_1+t+1)^{-2}),
$$
with $k$ being a real constant. Therefore, given any particular solution $x_1(t)$ of a particular instance of
the family of first-order Abel equations (\ref{LieFamily2}), the general solution, $x(t)$, of this instance is
$$
x(t)=\left((x_1(t)+t+1)^{-2}+k e^{-2t}\right)^{-1/2}-t-1.
$$

Note that our previous procedure can be straightforwardly generalized to derive common time-dependent superposition rules for generalized Abel equations \cite{Mo03}, i.e., the differential equations of the form
$$
\frac{dx}{dt}=a_{0}(t)+a_1(t)x+a_2(t)x^2+\ldots+a_n(t)x^n, \qquad n\geq 3.
$$
Actually, their study can be approached by analyzing the existence of two vector fields of the form
\begin{equation*}
\begin{aligned}
Y_1(t,x)&=(b_0(t)+b_1(t)x+\ldots+b_n(t)x^n)\frac{\partial}{\partial x},\\ Y_2(t,x)&=(b'_0(t)+b'_1(t)x+\ldots+b'_n(t)x^n)\frac{\partial}{\partial x},\qquad b'_n(t)\neq 0,
\end{aligned}
\end{equation*}
satisfying the relation $[\bar Y_1,\bar Y_2]=2(\bar Y_2-\bar Y_1)$ and following a procedure similar to the one developed above.

 \subsection{Lie families and second-order differential equations}
Common time-dependent superposition rules describe solutions of non-autonomous systems of first-order differential equations. Nevertheless, we shall now illustrate how this new kind of superposition rules can be applied to analyze also families of second-order differential equations. More specifically, we shall derive a common time-dependent superposition rule in order to express the general solution of any instance of a family of Milne--Pinney equations \cite{Ch40,Re99,ReIr01} in terms of each generic pair of particular solutions, two constants, and the time. In this way, we provide a generalization to the setting of dissipative Milne--Pinney equations of the expression previously derived to analyze the solutions of Milne--Pinney equations in \cite{CL08M}.

Consider the family of dissipative Milne--Pinney equations  \cite{Re99,ReIr01,Th52, PK07} of the form
\begin{equation}\label{InfFam}
 \begin{aligned}
 \ddot x&=-\dot F\dot x+\omega^2x+e^{-2F}x^{-3},
 \end{aligned}
\end{equation}
with a fixed time-dependent function $F=F(t)$, and parametrized by an arbitrary time-dependent function
$\omega=\omega(t)$. The physical motivation for the study of dissipative Milne--Pinney equations comes from its appearance in dissipative quantum mechanics \cite{Sr86,Ha75,Na86,NBA97}, where, for instance, their solutions are used to obtain Gaussian solutions of non-conservative time-dependent quantum oscillators \cite{Na86}. Moreover, the mathematical properties of the solutions of dissipative Milne--Pinney equations have been studied by several authors from different points of view as well as for different purposes \cite{CL08M,CGL08,CL08Diss,Re99,ReIr01,Wa68,Ha10,RC82}. As relevant instances, consider the works \cite{CL08Diss,Re99} which outline the state-of-the-art of the investigation of dissipative and non-dissipative Milne--Pinney equations. One of the main achievements on this topic (see \cite[Corollary 5]{Re99}) is concerned with an expression describing the general solution of a particular class of these equations in terms of a pair of generic particular solutions of a second-order linear differential equations and two constants. Recently the theory of quasi-Lie schemes and the theory of Lie systems enabled us to recover this latter result and other new ones from a geometric point of view \cite{CLR08c,CGL08}.

Note that introducing a new variable $v\equiv \dot x$, we transform the family (\ref{InfFam}) of second-order
differential equations into a family of first-order ones
\begin{equation}\label{LieFamily2}
\left\{\begin{aligned}
\dot x&=v,\\
\dot v&=-\dot Fv+\omega^2 x+e^{-2F}x^{-3},
\end{aligned}\right.
\end{equation}
whose dynamics is described by the following family of time-dependent vector fields on ${\rm T}\mathbb{R}$
parametrized by $\omega$,
\begin{equation*}
Y_\omega=\left(-\dot Fv+e^{-2F}x^{-3}+\omega^2x \right)\pd{}{v}+v\pd{}{x}, \qquad \omega\in
\Lambda=C^{\infty}(t).
\end{equation*}
Let us show that the above family is a Lie family whose common superposition rule can be used to analyze the
solutions of the family (\ref{InfFam}).

In view of Theorem \ref{MT}, if the family of systems related to the above family of time-dependent vector
fields is a Lie family, that is, it admits a common time-dependent superposition rule in terms of $m$
particular solutions, then the family of vector fields on $\mathbb{R}\times\mathbb{R}^{n(m+1)}$ given by
\begin{equation}\label{fam}\widetilde Y_\omega, \,\, [\widetilde Y_\omega,\widetilde Y_{\omega'}],\,\,
[\widetilde Y_\omega,[\widetilde Y_{\omega'},\widetilde Y_{\omega''}]],\,\,[\widetilde Y_\omega,
[\widetilde Y_{\omega'},[\widetilde Y_{\omega''},\widetilde Y_{\omega'''}]]]\ldots,\,\, \omega,\omega',\omega'',\omega''',\ldots\in \Lambda,
\end{equation}
spans an involutive generalized distribution with leaves of rank $r\leq n\cdot m+1$.

Note that the distribution spanned by all $\widetilde Y_\omega$ is generated by the vector fields $\widetilde
Y_1$ and $\widetilde Y_2$, with
\begin{equation*}
Y_1=\left(-\dot Fv+e^{-2F}x^{-3}+x \right)\pd{}{v}+v\pd{}{x}, \quad Y_2=\left(-\dot
Fv+e^{-2F}x^{-3}\right)\pd{}{v}+v\pd{}{x},
\end{equation*}
since $\widetilde Y_\omega= (1-\omega^2)\widetilde Y_2+\omega^2 \widetilde Y_1$.  It is easy to see that the
prolongation $[\widetilde Y_1,\widetilde Y_2]$ is not spanned by $\widetilde Y_1$ and $\widetilde Y_2$ and, so
that we have to include the prolongation $\widehat Y_3=[\widetilde Y_1,\widetilde Y_2]$ to the picture, where
$$
Y_3=x\pd{}{x}-(v+x\dot F)\pd{}{v}.
$$
In the case $m=0$, the distribution spanned by the vector fields, $\widetilde Y_1,\widetilde Y_2,\widehat
Y_3$, does not admit a non-trivial first-integral.  In the case $m>0$, the vector fields $\widetilde
Y_1,\widetilde Y_2,\widehat Y_3$ do not span all the elements of family (\ref{fam}) and we need to add to them
the prolongation  $\widehat Y_4=[\widetilde Y_1,[\widetilde Y_1,\widetilde Y_2]]$, with
$$
Y_4=(2v+x\dot F)\pd{}{x}+(2e^{-2F}x^{-3}-2x-\dot F(v+x\dot F)-x\ddot F)\pd{}{v}.
$$
The vector fields $\widetilde Y_1,\widetilde Y_2,\widehat Y_3,\widehat Y_4$ satisfy the commutation relations
$$
\begin{aligned}
\left[\widetilde Y_1,\widetilde Y_2\right]&=\widehat Y_3,\cr \left[\widetilde Y_1,\widehat
Y_3\right]&=\widehat Y_4,\cr \left[\widetilde Y_1,\widehat Y_4\right]&=(4+\dot F^2+2\ddot F)\widehat Y_3-(\dot
F\ddot F+\dddot F)(\widetilde Y_1-\widetilde Y_2),\cr \left[\widetilde Y_2,\widehat Y_3\right]&=2(\widetilde
Y_1-\widetilde Y_2)+\widehat Y_4,\cr \left[\widetilde Y_2,\widehat Y_4\right]&=(2+\dot F^2+2\ddot F)\widehat
Y_3-(\dot F\ddot F+\dddot F)(\widetilde Y_1-\widetilde Y_2),\cr
\left[\widehat Y_3,\widehat Y_4\right]&=-2\widehat Y_4-2(\widetilde Y_1-\widetilde Y_2) (4+\dot F^2+2\ddot F).\\
\end{aligned}
$$
Consequently, the vector fields $\widetilde Y_1,\widetilde Y_2,\widehat Y_3, \widehat Y_4$ span the vector
fields of the family (\ref{fam}). Adding $\widetilde Y_1$ to each prolongation of the previous set, that is,
considering the vector fields $\widetilde X_1=\widetilde Y_1$, $\widetilde X_2=\widetilde Y_2$, $\widetilde
X_3=\widetilde Y_1+\widehat Y_3$, and $\widetilde X_4=\widetilde Y_1+\widehat Y_4$, we get that the family of
time-prolongations, $\widetilde X_1,\widetilde X_2,\widetilde X_3,\widetilde X_4$, which spans the vector
fields of the family (\ref{fam}). The commutation relations among them read
$$
\begin{aligned}
\left[\widetilde X_1,\widetilde X_2\right]&=\widetilde X_3-\widetilde X_1,\cr \left[\widetilde X_1,\widetilde
X_3\right]&=\widetilde X_4-\widetilde X_1,\cr \left[\widetilde X_1,\widetilde X_4\right]&=-(\dot F\ddot
F+\dddot F+4+\dot F^2+2\ddot F)\widetilde X_1+(\dot F\ddot F+\dddot F)\widetilde X_2+(4+\dot F^2+2\ddot
F)\widetilde X_3,\cr \left[\widetilde X_2,\widetilde X_3\right]&=2\widetilde X_1-2\widetilde X_2-\widetilde
X_3+\widetilde X_4,\cr \left[\widetilde X_2,\widetilde X_4\right]&=-(1+\dot F^2+2\ddot F+\dot F\ddot F+\dddot
F)\widetilde X_1+ (\dot F\ddot F+\dddot F)\widetilde X_2+(1+\dot F^2+2\ddot F)\widetilde X_3,\cr
\left[\widetilde X_3,\widetilde X_4\right]&=-3 \widetilde X_4+(4+\dot F^2+2\ddot F)\widetilde X_3+(8+\dddot
F+\dot F\ddot F+2\dot F^ 2+4\ddot
F)\widetilde X_2+\\
+&(-9-3\dot{ F}^2-6\ddot F-\dot F\ddot F-\dddot F)\widetilde X_1.\cr
\end{aligned}
$$
As a consequence of Lemma \ref{Prol}, we get that the vector fields $\bar X_1$, $\bar X_2$ $\bar X_3$ and
$\bar X_4$ close on the same commutation relations as the vector fields $\widetilde X_1$, $\widetilde X_2$,
$\widetilde X_3$, $\widetilde X_4$. Hence, in view of Theorem \ref{MT}, the family (\ref{LieFamily2}) is a Lie
family and the knowledge of non-trivial first-integrals of the vector fields of the distribution $\mathcal{D}$
spanned by $\widetilde X_1$, $\widetilde X_2$, $\widetilde X_3$, $\widetilde X_4$ provides us with a common
time-dependent superposition rule.

Note that, as the vector fields $\widetilde X_1$, $\widetilde X_1-\widetilde X_2$ and their Lie brackets span
the whole distribution $\mathcal{D}$, a function $G:\mathbb{R}\times {\rm T}\mathbb{R}^3\rightarrow
\mathbb{R}$ is a first-integral for the vector fields of the distribution $\mathcal{D}$ if and only if it is a
first-integral for the vector fields $\widetilde X_1$ and $\widetilde X_2-\widetilde X_1$. Therefore, we can
reduce the problem of finding first-integrals for the vector fields of the distribution $\mathcal{D}$ to
finding common first-integrals $G$ for the vector fields $\widetilde X_1$ and $\widetilde X_1-\widetilde X_2$.

Let us analyze the implications of $G$ being a first-integral of the vector field
$$
\widetilde X_1-\widetilde X_2=\sum_{i=0}^2x_i\frac{\partial}{\partial v_i}.
$$
The characteristics of the above vector field are the solutions of the system
$$
\frac{dv_0}{x_0}=\frac{dv_1}{x_1}=\frac{dv_2}{x_2},\qquad dx_0=0,\quad dx_1=0,\quad dx_2=0,\quad dt=0,
$$
that is, the solutions are curves in $\mathbb{R}\times{\rm T}\mathbb{R}^3$ of the form
$s\mapsto(t,x_0,x_1,x_2,v_0(s),v_1(s),v_2(s))$, with  $\xi_{02}=x_0v_2(s)-x_2v_0(s)$ and
$\xi_{12}=x_1v_2(s)-x_2v_1(s)$ for two real constants $\xi_{02}$ and $\xi_{12}$. Thus, there exists a function
$G_2:\mathbb{R}^6\rightarrow\mathbb{R}$ such that $G(p)=G_2(t,x_0,x_1,x_2,\xi_{02},\xi_{12})$, with
$p\in\mathbb{R}\times {\rm T}\mathbb{R}^3$, $\xi_{02}=x_0v_2-x_2v_0$, and $\xi_{12}=x_1v_2-v_1x_2$. In other
words, $G$ is a function of $t,x_0,x_1,x_2,\xi_{02},\xi_{12}$.

The function $G$ also satisfies the condition $\widetilde X_1G=0$ which, in terms of the coordinate system
$\{t,x_0,x_1,x_2,\xi_{02} \xi_{12},v_2\}$, reads
\begin{multline*}
\widetilde X_1 G=\frac{\partial G}{\partial t}+\frac{(x_0v_2-\xi_{02})}{x_2}\frac{\partial G}{\partial
x_0}+\frac{(x_1v_2-\xi_{12})}{x_2}\frac{\partial G}{\partial x_1}+v_2\frac{\partial G}{\partial x_2}-\\-
\left[\dot F\xi_{12}+e^{-2F}\left(\frac{x_2}{x_1^3}-\frac{x_1}{x_2^3}\right)\right]\frac{\partial
G}{\partial\xi_{12}}-\left[\dot
F\xi_{02}+e^{-2F}\left(\frac{x_2}{x_0^3}-\frac{x_0}{x_2^3}\right)\right]\frac{\partial G}{\partial\xi_{02}}=0.
\end{multline*}
That is, defining the vector fields
$$
\begin{aligned}
\Xi_1&=\frac{\partial}{\partial t}-\frac{\xi_{12}}{x_2}\frac{\partial }{\partial x_1}-\frac{\xi_{02}}{x_2}\frac{\partial}{\partial x_0}
+\left[-\dot F\xi_{12}-e^{-2F}\left(\frac{x_2}{x_1^3}-\frac{x_1}{x_2^3}\right)\right]\frac{\partial }{\partial\xi_{12}}\\
&+\left[-\dot F\xi_{02}-e^{-2F}\left(\frac{x_2}{x_0^3}-\frac{x_0}{x_2^3}\right)\right]\frac{\partial }{\partial\xi_{02}},\\
\Xi_2&=\frac{x_0}{x_2}\frac{\partial }{\partial x_0}+\frac{x_1}{x_2}\frac{\partial }{\partial x_1}+\frac{\partial}{\partial x_2},\\
\end{aligned}
$$
the condition $\widetilde X_1G=0$ implies that $\Xi_1G_2+v_2 \Xi_2G_2=0$ and, as $G_2$ does not depend on
$v_2$, the function $G$ must be simultaneously a first-integral for $\Xi_1$ and $\Xi_2$, i.e., $\Xi_1G=0$ and
$\Xi_2G=0$.

Applying again the method of characteristics to the vector field $\Xi_2$, we get that $F$ can depend just on
the variables $t,\xi_{02},\xi_{12},\Delta_{02}=x_0/x_2$ and $\Delta_{12}=x_1/x_2$, that is, there exists a
function $G_3:\mathbb{R}^5\rightarrow\mathbb{R}$ such that
$G(t,x_0,x_1,x_2,v_0,v_1,v_2)=G_2(t,x_0,x_1,x_2,\xi_{02},\xi_{12})=G_3(t,\xi_{02},\xi_{12},\Delta_{02},\Delta_{12})$.

We are left to check out the implications of the equation $\Xi_1G=0$. Using the coordinate system
$\{t,\xi_{02},\xi_{12},\Delta_{02},\Delta_{12},v_2,x_2\}$ and taking into account that
$G(t,x_0,x_1,x_2,v_0,v_1,v_2)=G_3(t,\xi_{02},\xi_{12},\Delta_{02},\Delta_{12})$, the previous equation can be
cast into the form $\Xi_1G=\frac{1}{x_2^2}\Upsilon_1G_3+\Upsilon_2G_3=0$, where
$$
\begin{aligned}
\Upsilon_1&=\sum_{i=0}^1\left(-\xi_{i2}\frac{\partial }{\partial \Delta_{i2}}-e^{-2F}\left(\Delta_{i2}^{-3}-
\Delta_{i2}\right)\frac{\partial }{\partial\xi_{i2}}\right),\\
\Upsilon_2&= -\dot F\xi_{12}\frac{\partial}{\partial \xi_{12}}-\dot F\xi_{02}\frac{\partial}{\partial
\xi_{02}}+\frac{\partial}{\partial t}.
\end{aligned}
$$
As $G_3$ depends on the variables $t,\Delta_{02},\Delta_{12},\xi_{12},\xi_{02}$ only, we have that
$\Upsilon_1G=0$ and $\Upsilon_2G=0$. Repeating {\it mutatis mutandis} the previous procedures in order to
determine the implications of being a first-integral of $\Upsilon_1$ and $\Upsilon_2$, we finally get that the
first-integrals of the distribution $\mathcal{D}$ are functions of $I_1,I_2$ and $I$, with
$$
I_i=e^{2F}(x_0 v_i-x_iv_0)^2+\left[\left(\frac{x_0}{x_i}\right)^2+\left(\frac{x_i}{x_0}\right)^2\right],\qquad
i=1,2,
$$
and
$$
I=e^{2F}(x_1v_2-x_2v_1)^2+\left[\left(\frac{x_1}{x_2}\right)^2+\left(\frac{x_2}{x_1}\right)^2\right].
$$
Defining $\bar v_2=e^{F}v_2, \bar v_1=e^{F}v_1$ and $\bar v_0=e^Fv_0$, the above first-integrals read
$$
I_i=(x_0\bar v_i-x_i\bar
v_0)^2+\left[\left(\frac{x_0}{x_i}\right)^2+\left(\frac{x_i}{x_0}\right)^2\right],\qquad i=1,2,
$$
and
$$
I=(x_1\bar v_2-x_2\bar v_1)^2+\left[\left(\frac{x_1}{x_2}\right)^2+\left(\frac{x_2}{x_1}\right)^2\right].
$$
Note that these first-integrals have the same form as the ones considered in \cite{CLR08c} for $k=1$.
Therefore, we can apply the procedure done there to obtain that
\begin{equation}\label{Super}
x_0=\sqrt{k_1x_1^2+k_2x_2^2+
2\sqrt{\lambda_{12}[-(x_1^4+x_2^4)+I\,x_1^2x_2^2\,]}}\,,\\
\end{equation}
with $\lambda_{12}$ being a function of the form
\begin{equation*}
\lambda_{12}(k_1,k_2,I)= \frac{k_1k_2 I+(-1+k_1^2+k_2^2)}{I^2-4},
\end{equation*}
and where the constants $k_1$ and $k_2$ satisfy special conditions in order to ensure that $x_0$ is real
\cite{CL08M}.

Expression (\ref{Super}) permits us to determine the general solution, $x(t)$, of any instance of family
(\ref{InfFam}) in the form
\begin{equation}\label{Super2}
x(t)=\sqrt{k_1x_1^2(t)+k_2x_2^2(t)+
2\sqrt{\lambda_{12}[-(x_1^4(t)+x_2^4(t))+I\,x_1^2(t)x_2^2(t)\,]}}\,,\\
\end{equation}
with
$$
I=e^{2F(t)}(x_1(t)\dot x_2(t)-x_2(t)\dot
x_1(t))^2+\left[\left(\frac{x_1(t)}{x_2(t)}\right)^2+\left(\frac{x_2(t)}{x_1(t)}\right)^2\right],
$$
in terms of two of its particular solutions, $x_1(t)$, $x_2(t)$, its derivatives, the constants $k_1$ and
$k_2$, and the time (included in the constant of the motion $I$).

Note that the role of the constant $I$ in expression (\ref{Super2}) differs from the roles carried out by
$k_1$ and $k_2$. Indeed, the value of $I$ is fixed by the particular solutions $x_1(t)$, $x_2(t)$ and its
derivatives, while, for every pair of generic solutions $x_1(t)$ and $x_2(t)$, the values of $k_1$ and $k_2$
range within certain intervals ensuring that $x(t)$ is real.

It is clear that the method illustrated here can also be applied to analyze solutions of any other family of second-order differential equations related to a Lie family by introducing the new variable $v=\dot x$. Additionally, it is worth noting that in the case $F(t)=0$ the family of dissipative Milne--Pinney equations (\ref{InfFam}) reduces to a family of Milne--Pinney equations appearing broadly in the literature (see \cite{AL08} and references therein), and the expression (\ref{Super2}) takes the form of the expression obtained in \cite{CL08M} for these equations.

\section{Conclusions and Outlook.}
\noindent

We have proposed a generalization of Lie Theorem in order to characterize those families of non-autonomous systems of first-order ordinary differential equations, the so-called Lie families, that admit a common time-dependent superposition rule. We have studied the relations of quasi-Lie systems and Lie families.

In order to illustrate the usefulness of our achievements, we have derived common time-dependent superposition rules for studying dissipative Milne--Pinney equations and Abel equations. In the case of Abel equations, our result expresses the general solution of any particular instance of a Lie family of non-autonomous Abel equations in terms of each generic particular solution, a constant, and the time. In this way, we have initiated a new approach to study the solutions of these equations. Additionally, it is worth noting that the analyzed Lie family of Abel equations contains an autonomous instance admitting a special kind of superposition rule derived by Chiellini \cite{Ch40}. Unlike such a special superposition rule, our common superposition rule does not require the use of particular solutions obeying any kind of extra condition and, therefore, it clearly represents an improvement with respect to Chiellini's technique.

We have shown how common time-dependent superposition rules can be used to analyze second-order differential equations by means of the study of a family of Milne--Pinney equations. More specifically, we have derived a common-superposition rule allowing us to obtain the general solution of any instance of such a family in terms of a generic pair of its particular solutions, their derivatives in terms of the time, and the time. Such an expression represents an interesting improvement with respect to previous results and methods, as it generalizes the superposition rule given in \cite{CL08M} for the usual Milne--Pinney equations to the dissipative case.

We hope to get in the future new results on the theory of common time-dependent superposition rules and, additionally, to describe new applications, where our achievements can be used.
\section*{Acknowledgements}
 Partial financial support by research projects MTM2009-11154 and E24/1 (DGA)
 are acknowledged. Research of the second author financed by the Polish Ministry of Science and Higher Education under the
 grant N N201 365636.

\end{document}